\documentclass[12pt]{article}

\usepackage[utf8]{inputenc}

\usepackage{amssymb}
\usepackage{latexsym}
\usepackage{exscale}

\usepackage{tikz}

\usetikzlibrary{trees}

\usepackage{amssymb, amsmath}
\usepackage{enumerate}
\usepackage{enumerate}
\usepackage[top=1in,left=1in,right=1in,bottom=1in]{geometry}
\usepackage{ulem}

\def \beq{\begin{equation}}
\def \eeq{\end{equation}}

\renewcommand{\rq}[1]{(\ref{#1})}
\newtheorem{lemma}{Lemma}
\newtheorem{prop}{Proposition}
\newtheorem{thm}{Theorem}
\newtheorem{cor}{Corollary}
\newcommand{\x}{{ x}}
\newcommand{\y}{{y}}

\newcommand{\bR}{{ \mathbb R  }}
\newcommand{\bC}{\Bbb C}
\newcommand{\bZ}{\Bbb Z}

\newcommand{\bN}{\mathbb{N}}
\newcommand{\bK}{\Bbb K}

\newcommand{\la}{\mbox{$\lambda$}}

\newcommand{\pa }{\partial }
\newcommand{\f}{\varphi}

\newcommand{\ep}{\epsilon}

\newcommand{\al}{\alpha }

\newcommand{\La}{\Lambda }

\newcommand{\om}{{\omega}}

\def\<{\langle} \def\>{\rangle}

\newcommand{\A}{\tilde{A^{\alpha}}}

\title{A symmetry formula for the spectral fractional Laplacian, and applications to
 boundary controllability for plate equation with structural damping.}
\author{Sergei Avdonin\footnote{The research of S.A. was  supported  in part by the National Science Foundation, grant DMS 2308377, and by the Ministry of Education and Science of the Russian Federations part of the program of the Moscow Center for Fundamental and Applied Mathematics under the Agreement No. 075-15-2025-345.}, and\footnote{Sergei Ivanov, an active participant in this project, passed away in February.}  Julian Edward\footnote{corresponding author} }
\date{May 2026}
\begin{document}

\maketitle

Sergei Avdonin: Department of Mathematics and Statistics, University of Alaska Fairbanks, Fairbanks, 99775, Alaska, USA, saavdonin@alaska.edu
          
Julian Edward:  Department of Mathematics and Statistics, Florida International University, Miami, FLorida 33199, USA, edwardj@fiu.edu

\vspace{9in}
\begin{abstract}
Let\footnote{Key words: vibrating plate, structural damping, spectral fractional Laplacian, control.} \footnote{MSC numbers: 93B05, 26A33, 74K10, 93C20
} $\Delta$  be the Dirichlet Laplacian on a  bounded domain $\Omega \subset \bR^{N}$, and let $(-\Delta)^\al$ be the associated spectral fractional Laplacian with $\al \leq 1, \ \rho <2$. For general bounded domains with $C^2$ boundary, we prove a symmetry formula for $\al <1/2$, extending a  result previously proven on rectangles for $\al <1$. As a consequence of this formula, well-posedness results are proven for the structurally damped plate equation 
$$u_{tt}+\Delta^2u+(-\Delta)^\al u_t=0$$ 
subject to Dirichlet or moment boundary control. 
For rectangular domains with $\al <1$, we  prove boundary null-controllability results.
 For $\al <1/2, \ \rho \leq 2$, Dirichlet null controllability is proved for the unit disk in $\bR^2$. This analysis then extended to the classical case, $\al =1$, on  rectangles, where  higher regularity is required for Dirichlet control.

\end{abstract}

\section{Introduction}

Let $\Omega \subset \bR^N$ be a bounded. We will assume either that $\Omega$ is a rectangle or that the boundary  $\Gamma$ is $C^2$. Let $\Gamma_1$ be an open subset of the boundary, and 
$\Gamma_2$ the interior of its complement in the boundary. 
Let $T>0$, and let $Q=\Omega \times (0,T)$, $\Sigma= \pa \Omega \times (0,T)$, and 
$\Sigma_j= \Gamma_j \times (0,T)$, $j=1,2$.

Let $A=-\Delta$ with operator domain 
$H^2(\Omega )\cap H^1_0(\Omega )$. Then it is well known that $A$ is a positive, self-adjoint operator. We will denote by $X^p$ the operator domain of $A^{p/2}$, and $H^p(\Omega)$ the standard $p$-Sobolev space on $\Omega$.

We wish to discuss properties of the vibrating plate, modelled by 
\begin{eqnarray}
u_{tt}+A^2 u+\rho A^\al u_t & = &0, \mbox{ on } Q,\label{beam2o}\\
u|_{ \Sigma_2}=\Delta u|_{ \Sigma_2}& = & 0\\
u|_{ \Sigma_1 }& =& f\\
\Delta u|_{ \Sigma_1 }& =& g\\
u(*,0)=u^0(*),\ u_t(*,0)& = & u^1(*).\label{init2o}
\end{eqnarray}
This system is actuated through  control mechanisms prescribed $f,g$. Throughout this paper, controllability will always mean
the ability of steering any initial state $(u^0,u^1)$ to zero over a finite time by some appropriate
input functions $f,g$ (i.e. exact controllability to zero or null controllability).

The term $(-\Delta)^\al u_t$  models a specific dissipative effect, known as structural damping, when $\al \in (0,2)$. To the best of our knowledge, this was introduced in \cite{CR} assuming $\al =1$:
“The basic property of structural damping, which is said to be consistent with empirical studies,
is that the amplitudes of the normal modes of vibration are attenuated at rates which are proportional to the oscillation frequencies.” This model was also studied under the name “proportional
damping” (cf. \cite{Ba}). The quite different case $\al=2$ is known as “Kelvin–Voigt” damping. In the case of a control distributed in the interior and $\al \in (0,2]$, this is the first class of parabolic-like control models considered
in \cite{LT},\cite{T2},  
also see \cite{AL}.

Until  recently, very little was known about the well-posedness of the system above for non-integer $\al$ when $f\neq 0$. One obstruction for proving the existence of a weak solution was lack of an ``integration by parts" formula that would enable us to define a weak solution. Progress was made recently in \cite{E2}, where rectangular domains were considered. There, it is shown that for $\al \in (0,1)$ and  any rectangle $\Omega =R$,  
a formula due to Song and Vondracek, \cite{SV}, gives rise to 
a natural extension of $A^\al$ , which we label $\A$, satisfying 
\beq
\int_R \A u(x)v(x)dx=\int_R u(x)A^\al v(x)dx, \forall u\in H^{N+3/2}(R), v\in X^{N+3/2}. \label{symm2a}
\eeq
The definition of $\A$ will be given in Section 2.1. For more on the properties of $A^\al$, the reader is referred to \cite{E2}.
We then interpret the plate equation \rq{beam2o} as
$$u_{tt}+A^2 u+\rho \A u_t  = 0.$$
As a consequence,
we showed in \cite{E2}:
\begin{prop}\label{wp0}
Let $T>0$. Let $\Omega =R$, with $R$ a rectangle.
Let $\al\in (0,1)$. 

A) Suppose $g=0$. Then, given $f\in L^2(\Sigma_1)$, there exists a unique solution, $u$,  to \rq{beam2o}-\rq{init2o}, and it satisfies $u\in C^0(0,T;X^{-1})\cap C^1(0,T;X^{-3})$.

B) Suppose $f=0$. Then, given $g\in L^2(\Sigma_1)$, there exists a unique solution, $u$,  to \rq{beam2o}-\rq{init2o}, and it satisfies $u\in C^0(0,T;X^{1})\cap C^1(0,T;X^{-1})$.
\end{prop}

Well-posedness follows from this, a Fourier analysis of the adjoint problem, and a standard duality argument. 

We note  the absence of boundary terms in \rq{symm2a}, in contrast to the classical formula for $\al =1$. The consequences of this difference will become clear in our well-posedness statements 
for the ``classically structurally damped" plate equation below.

One of the purposes of this paper is to extend \rq{symm2a} to general $\Omega$ provided $\al <1/2$. Let $\delta (x)$ be the distance between $x$ and $\pa \Omega$.
\begin{thm}\label{ibpd}
Let $\Omega \subset \bR^N$ be a bounded, $C^{2}$ domain.  Let $\al <1/2$. 
Let $u\in H^{3+N/2}(\Omega), w\in X^{3+N/2}.$ 
Then 

i) $\A u \in  L^1(\Omega)$.

ii)  Let $ \tilde{\ep}>0$. 
Then $u \mapsto \delta(\x )^{2\al +\tilde{\ep}}\A u$ is a continuous mapping from $H^{3+N/2}(\Omega )$ to $C(\overline{\Omega})$.

iii)
\beq 
\int_\Omega\A u (\x)\ w(\x)\ d\x =\int_\Omega u(\x)\ A^{\al}w(\x)\ d\x.\label{symm1}
\eeq

\end{thm}
\begin{cor}\label{wp<1}
Let $T>0$.
Let $\Omega, \al$ be as in Theorem \ref{ibpd}. 

A) Suppose $g=0$. Then, given $f\in L^2(\Sigma_1)$, there exists a unique solution, $u$,  to \rq{beam2o}-\rq{init2o}, and it satisfies $u\in C^0(0,T;X^{-1})\cap C^1(0,T;X^{-3})$.

B) Suppose $f=0$. Then, given $g\in L^2(\Sigma_1)$, there exists a unique solution, $u$,  to \rq{beam2o}-\rq{init2o}, and it satisfies $u\in C^0(0,T;X^{1})\cap C^1(0,T;X^{-1})$.
\end{cor}
We now discuss the classical  case $\al =1$.
For Dirichlet control, 
the energy spaces must change because of the boundary term arising in classical integration by parts. We illustrate this with the one-dimension beam, thus $\Omega =(0,\pi)$.
Thus consider the system 
\begin{eqnarray}
u_{tt}+A^2 u+\rho A u_t& = & 0,\label{beamD}\\
u(0,t)=u_{xx}(0,t) & = & 0,\label{bc19}\\
u(\pi,t) & = & f(t),\label{bc29}\\ 
u_{xx}(\pi,t) & = & g(t),\label{bc39}\\
u(x,0)=0,\ u_t(x,0)& = & 0.\label{initD}
\end{eqnarray}
Denote the normalized eigenfunctions of $A$ by $\f_n(x)=\sin (nx)/\sqrt{\pi}$.
Let $w(x,t)=\f_n(x)e^{\la^+_n(T-t)}$,  where $\la_n^+$ is a member of the frequency spectrum, whose formula will be given in \rq{lapm}. Then
\beq 
w_{tt}+A^2w-\rho Aw_t=0.\label{w}
\eeq
Write $u(x,t)=\sum_{n=1}^\infty a_n(t)\f_n(x)$.
Integrating the left hand side of \rq{w} against $u$, we get
$$
a_n'(T)+(\rho n^2+\la_n^+)a_n(T)=-(-1)^n\rho nf(T)
+\frac{(-1)^n}{\sqrt{\pi}}\int_0^Te^{\la_n^+(T-t)}\big (-n^3f(t)+ng(t)-\rho nf'(t)\big )dt.
$$
It is clear from this formula that if $f,g\in L^2(0,T)$, then $a_n'$ is not necessarily continuous. Note that for moment control ($f=0$), this issue does not arise. This difference between Dirichlet and moment control was already understood in \cite{T1}.

Thus,
in order to state an exact controllabilty result for the plate, we need to consider controls more regular that $L^2$. To this end, we define 
$$H_*^2(0,T;Y)=\{ f\in H^2(0,T;Y): f(\cdot ,0)=f'(\cdot ,0)=0\}$$
and 
$$
C_*^2(0,T;Y)=\{ f\in C^2(0,T;Y): f(\cdot ,0)=f'(\cdot, 0)=0\};
$$
here $Y$ is some vector space that will typically be $L^2(\Sigma_1)$. 

For simplicity of exposition, we restrict our domains to a class of product spaces. Suppose $\Omega =(0,\pi)\times M$, where $M\subset \bR^{N-1}$  a  bounded, $C^2$ domain. We assign coordinate $x$ to $(0,\pi)$, and $y=(y_1,...,y_{N-1})$ to $M$. 
\begin{prop}\label{wp1}
Let $T>0$.
Suppose $\Omega =(0,\pi)\times M$, with $M$ as above.
Assume $\Gamma_1$ is an open subset of the face $x=\pi$, and $\Gamma_2$ is the interior of $\Gamma \setminus \Gamma_1$.

A) Suppose $g=0$. Let 
$
f\in H^2_*(0,T;L^2(\Gamma_1).$
Then there exists a unique solution, $u$,  to \rq{beam2o}-\rq{init2o}, and it satisfies $u\in C^0(0,T;X^{3})\cap C^1(0,T;X^{1})$.

B) Suppose $f=0$. Suppose
$g\in L^2(\Sigma_1).$
Then there exists a unique solution, $u$,  to \rq{beam2o}-\rq{init2o}, and it satisfies $u\in C^0(0,T;X^{1})\cap C^1(0,T;X^{-1})$.
\end{prop}
A version of part B of this  proposition was proven in \cite{T1} for general domains and more optimal energy spaces, but our version allows us to use the Fourier method to prove null-controllability arguments. 

We are now able to state our controllability results.
\begin{thm}\label{prod}
Let $\al \in (0,1),$ and $\rho <2.$
Suppose $\Omega =R$ is a rectangle.  
Let $T>0.$ 

A) 
Set $g=0$.
Given $(u^0,u^1)\in (X^{-1}\times X^{-3})$, there exist $f\in L^2(\Sigma_1)$ such that the solution $u$ to the system \rq{beam2o}-\rq{init2o} solves
$$u(x,T)=u_t(x,T)=0,
$$
with 
$$
\| f\|_{L^2(\Sigma_1)}\leq Ce^{C/T}(\| (u^0)\|_{X^{-1}}+\|u^1\|_{X^{-3}}).
$$

B) Set $f=0$.
Given $(u^0,u^1)\in (X^{1}\times X^{-1})$, there exists $g\in L^2(\Sigma_1)$ such that the solution $u$ to the system \rq{beam2o}-\rq{init2o} solves
$$u(x,T)=u_t(x,T)=0,
$$
with 
$$
\| g\|_{L^2(\Sigma_1)}\leq Ce^{C/T}(\| u^0\|_{X^{-1}}+\|u^1\|_{X^{-3}}).
$$

Here $C>0$ is a constant depending only on $\al,\rho$.
\end{thm}

\ 

\begin{thm}\label{prod1}
Let $\al =1,$ and $\rho <2.$ Let $\Omega$, $\Gamma_1,\Gamma_2$ be as in Proposition \ref{wp1}.
Let $T>0.$ 

A) Set $g=0$.
Given $(u^0,u^1)\in X^{3}\times X^{1}$, there exist $f\in H_0^2(0,T;L^2(\Gamma_1))$ 
such that the solution $u$ to the system \rq{beam2o}-\rq{init2o} solves
$$u(x,T)=u_t(x,T)=0,
$$
with 
$$
\| f''\|_{L^2(\Sigma_1)}\leq Ce^{Q(T)}(\| u^0\|_{X^{3}}+\|u^1\|_{X^{1}}).
$$

B) Set $f=0$.
Given $(u^0,u^1)\in X^{1}\times X^{-1}$, there exist $g\in L^2(\Sigma_1)$  such that the solution $u$ to the system \rq{beam2o}-\rq{init2o} solves
$$u(x,T)=u_t(x,T)=0,$$with 
$$\| g\|_{L^2(\Sigma_1)}\leq Ce^{Q(T)}(\| u^0\|_{X^{1}}+\|u^1\|_{X^{-1}}).$$

Here $C>0$ is a constant depending only on $\al,\rho$.
\end{thm}
For a discussion of the case $\al=1$ for the beam equation, see \cite{AE}.

Next, we 
 consider the unit disk,  $D\subset \bR^2$. For brevity, we consider only Dirichlet control.
\begin{thm}\label{prod3}
Let $\Omega =D$, and $\Gamma_1$ a relatively open subset of $\Gamma=S^1$. 
Let $\al \in (0,1/2)$.
Let $T>0.$ Set $g=0$. 
Then
given $(u^0,u^1)\in (X^{-1}\times X^{-3})$, there exist $f\in L^2(\Sigma_1)$ 
such that the solution $u$ to the system \rq{beam2o}-\rq{init2o} solves
$$u(x,T)=u_t(x,T)=0,
$$
with 
$$
\| f\|_{L^2(\Sigma_1)}\leq Ce^{C/T}(\| u^0\|_{X^{-1}}+\|u^1\|_{X^{-3}}).
$$
Here the constant $C$  depends only on $\al,\rho$.
\end{thm}
To prove the Theorems \ref{prod} and \ref{prod3}, we use the Fourier Method to prove an observability estimate.
An important ingredient in these proofs is an estimate, that might be of independent interest, that we now state  for rectangles in the notation of Theorem \ref{prod}. 
Let the rectangle $S\subset \tilde{R}$ be an inclusion of rectangles, with 
\beq
S=\{ y=(y_1,...,y_N): \ y_j\in (a_j,b_j)\} \mbox{ and }
\tilde{R}=\{ y=(y_1,...,y_N): \ y_j\in (0,l_j)\}, \label{rr}
\eeq
with $0\leq a_j<b_j\leq l_j$. Observe that the normalized eigenfunctions of the Dirichlet Laplacian on $\tilde{R}$ are
$$
\phi_{n_1,\ldots ,n_N}(y)=\big(\prod_{j=1}^N\frac{2}{l_j}\big)^{1/2}\sin(\frac{n_jy_j\pi}{l_j}),\ n_j\in \bN,
$$
with corresponding eigenvalues $$\kappa_{n_1,\ldots ,n_N} = \sum_{j=1}^N (\frac{\pi n_j}{l_j})^2.$$
It will be convenient to reparametrize  as $\{ \phi_m,\kappa_m: \ m\in \bN \}$, with $\kappa_m$ listed in non-decreasing order.
\begin{prop}\label{UM}
Let $S\subset \tilde{R}$. Then for any $\ell^2$ sequence $\{ c_m: m\in \bN\} $, 
$$
\int_S|\sum_{m\in \bN}c_m\phi_m(y)|^2\ dy\geq 
C\sum_{m\in \bN}|c_m|^2/m^4,
$$
where the constant $C>0$ depends on $S$.
\end{prop}
We are uncertain this result is new, but we were unable to find it in the literature. We will compare it to a result that appears in \cite{LZ} in the next section.

This paper is organized as follows. In the next subsection, we compare our results with the literature. In Section 2.1, prove Theorem \ref{ibpd}, and use it in Section 2.2 to prove Corollary \ref{wp<1}. In Section 2.3, we define the frequency spectrum and present a Fourier series solution to the adjoint problem corresponding to all our control problems. 
In Section 2.4, we apply results from \cite{AIS}, refined in \cite{AEI}, to prove the existence of a biorthogonal family of functions, satisfying an exponential estimate, and also to prove a complex ``windows" estimate; these will be used to prove Theorems 2, 3, 4.
  Theorem \ref{prod} is proven in Section 3, Theorem \ref{prod3} in Section 4, 
and Theorem \ref{prod1} in Section 5. Finally, in the appendix, we prove Proposition \ref{UM} and an analogue for the family $\{ 1, \sin (n\theta),\cos (n\theta )\}$, which is required for control on the disk.

\subsection{Literature review}

First, we
discuss well-posedness for the plate and beam equations with spectral fractional structural damping and non-homogeneous boundary conditions. One the the standard ways for proving well-posedness for evolution equations is to first recast them in their weak form, using integration by parts, see for instance \cite{lions1}.
For the case $\al =1$, where classical integration by parts is possible, the well-posedness and regularity for general Euclidean domains are addressed in  Triggiani's work in \cite{T},  which also discusses non-homogeneous Neumann boundary conditions, and in \cite{T1}.

Our Theorem \ref{ibpd} was first proven in \cite{E2} for rectangles. In \cite{AEI2}, an anologue is proven for  Neumann boundary conditions on the interval $(0,\pi)$. For more discussion on the spectral fractional Laplacian, the reader is referred to \cite{SV}, \cite{AD}, \cite{CS}.

Let $R$ be a rectangle, and 
let $\pa/\pa \eta$ be the outward pointing normal derivative at $\pa R$. 
For  the structurally damped plate equation with boundary control 
$$\frac{\pa^2u}{\pa \eta^2}(x,t)=f(x,t),\ x\in \pa R$$
 with $\al \in (0,1/2)$, regularity results are proven in \cite{H2}. It is unclear whether the methods of that paper would apply in our setting.

Regarding boundary null controllability, 
Miller \cite{mil} proved null controllability for  \rq{beam2o} in the case where $\al =1,\ f=0$, and 
$\Omega=M \times (0,\pi)$, where $M$ is a smooth, complete Riemannian manifold. It was assumed the control could be supported throughout the boundary face 
$\{ \pi\} \times M$. To prove this result, Miller first observed that boundary null controllability for the beam equation, i.e. $\Omega =(0,\pi)$, can be proved for moment control by the methods of \cite{AIS}. Then, it was argued that the controllability cost of a system is not
increased by taking its tensor product with a contraction semigroup. It would seem difficult to extend Miller's methods to the case where the control was supported in a proper subset of $\{\pi\} \times M$, as addressed here. It should be noted that 
the main focus of his paper, but not ours, was in obtaining sharp estimates on the cost of the control as $T\to 0^+$. For more discussion of small $T$ control costs, the reader is referred to \cite{AEI} and references therein. Other papers that consider boundary null controllability for plates 
include \cite{H} and \cite{AIS}, but in both papers the authors assumed  $\al =1$ and the controls chosen were not Dirichlet controls. In \cite{AE}, an analog for Theorem \ref{prod1} for beams is proven for Dirichlet control.

Miller also considers interior null controllability, where it is easier to relax the assumption $\al =1$.
A key tool in his arguments is
the inequality 
\beq
\int_{\Omega}\left|\sum_{j\leq \om_k}c_j\phi_j(x)\right|^2\geq C_1e^{-C_2\om_k} \sum_{j\leq \om_k}|c_j|^2,\label{lz}
\eeq
along with the natural damping properties of the system. Here $\{ \om_j, \phi_j\}$ are the spectrum and orthonormal eigenfunctions for the Laplacian. The inequality above appears first to have been proven in \cite{LZ} using a Carleman estimate due to Lebeau and Robbiano, \cite{LR}. Compared to our Proposition \ref{UM}, \rq{lz} holds for much more general domains. However, the exponential weight that appears on the right hand side seems to be vanish too quickly to prove the observability inequalities in this paper. For classical versions of \rq{lz}, known as Ingham type inequalities, the reader is referred to \cite{Y},\cite{E}.

Other results on  interior control include \cite{LT},\cite{AL}, \cite{E}, \cite{ET}, and \cite{AEI}. In \cite{mit}, under the assumption of periodic boundary conditions, controllability is proven via a Carleman estimate, assuming $\al=1$.

Finally, we should note that there are many different versions of the ``fractional Laplacian", which are not equivalent to our spectral fractional Laplacian, see \cite{L}, \cite{APR}, \cite{E2}, and references therein.

\section{Spectrum, well-posedness, and biorthogonal functions}

\subsection{Proof of Theorem 1}
In this section, we extend the results from \cite{E}, which treated only rectangular domains,  to general $C^2$ bounded domains. We were able to prove this extension only under the strengthened assumption $\al <1/2$.

The extension of $A^\al$ we consider is based on the following integral representation of the operator due to Song and Vondracek, \cite{SV}.
Let $p(t,x,y)$ be the heat kernel
associated to the Dirichlet Laplacian on $\Omega$. 
Thus
\begin{eqnarray*}
p_t-\Delta p & = & 0, \mbox{ on }Q,\\
p & = & 0,\mbox{ on } \Sigma,\\
p(0,x,y) & = & \delta (x-y), x,y\in \Omega.
\end{eqnarray*}
Then for $\al <1$ and $w\in C_0^\infty(R)$ (so $w$ vanishes at the boundary),
\beq
A^{\al}w(x)=PV \int_\Omega (w(x)-w(y))J(x,y)dy+ \kappa(x)w(x).\label{int}
\eeq
Here PV stands for principle value, and 
$$
J(x,y)=\frac{\al}{\Gamma(1-\al)}\int_0^\infty p(t,x,y)t^{-\al -1}dt,$$
$$
 \kappa(x) =\frac{\al}{\Gamma(1-\al)}\int_0^\infty \big (1-\int_0^\pi p(t,x,y)dy\big )t^{-\al -1}dt.
$$
As we will see below, $J$ is singular for $x=y$, which is why the principle value formulation is needed. The formula \rq{int} was proven in \cite{SV} using probabilistic methods, but a non-probabilistic proof is given in \cite{AD}. It should also be noted that the formula holds only for $\al <1$; note, for instance, that for $\al =1$ we would have $\Gamma (0)$.

We use \rq{int} to extend $A^\al$ to functions that don't vanish on the boundary. We denote this extension by $\A$.
We restate Theorem \ref{ibpd} for the reader's convenience.
\begin{thm}
 Let $\Omega \subset \bR^N$ be a bounded, $C^{2}$ domain. Assume $\al \in (0,1/2)$, and let $ \tilde{\ep}>0$.
Let $w\in X^{3+N/2},\, u\in H^{3+N/2}(\Omega ) $ and  $
\delta(\x):=dist(\x,\pa \Omega).$ Then

i) $\A u \in  L^1(\Omega)$; 

ii)  $u \mapsto \delta(\x )^{2\al +\tilde{\ep}}\A u$ is a continuous mapping from $H^{3+N/2}(\Omega )$ to $C(\overline{\Omega})$.

iii)
$$
\int_\Omega\A u (\x)\ w(\x)\ d\x =\int_\Omega u(\x)\ A^{\al}w(\x)\ d\x.
$$
\end{thm}
{\bf Remark}
The choice of Sobolev spaces for this theorem ensures that $u,w\in C^2(\overline{\Omega})$, by the Sobolev Imbedding Theorem,  and furthermore, for $w\in X^{3+N/2}$, there exists a positive constant $c$ such that  
\beq 
w(\x )<{c} \delta (\x),\ \forall {\bf x}\in \Omega .\label{bdec}
\eeq 

\

\noindent{\bf Proof of theorem:}

We first discuss the term involving $\kappa$.
It is shown in \cite{SV} that
 $\kappa$ is continuous on $\Omega$ and 
$$\kappa(\x)\asymp \delta(\x)^{-2\al}.$$ 
Thus, $u\mapsto \delta^{2\al+\tilde{\ep}}\kappa u$ is clearly continuous from $H^{3+N/2}(\Omega)$ to $C(\overline{\Omega}).$
Also, when $\al <1/2$, we have $\kappa u\in L^1(\Omega )$. 

We now prove our results for the terms involve $J$. First, we clarify and reformulate the principal value integral.
Let 
$$B_\ep(\x)=\{ \y\in \Omega: |\y-\x|<\ep\},
$$
For any $\x\in \Omega,$
$$
PV \int_\Omega (w(\x)-w(\y))J(\x,\y)\ d\y=\lim_{\ep \to 0}
\int_{\Omega\setminus B_\ep(\x)}(w(\x)-w(\y))J(\x,\y)\ d\y.$$
We reformulate this slightly. 
For $\ep>0$,
let 
$$U_\ep =\{ (\x,\y)\in \Omega\times \Omega: |\x-\y|\geq \ep\},$$
and let $J_\ep(\x,\y)=\chi_{U_\ep}(\x,\y)J(\x,\y).$
Then we define 
$$
PV \int_\Omega(w(\x)-w(\y))J(\x,\y)d\y=\lim_{\ep \to 0}
\int_\Omega(w(\x)-w(\y))J_\ep(\x,\y)d\y.
$$
We
 mention some properties of $J.$
Since $p(t,\x,\y)=p(t,\y,\x)$, it follows that $$J(\x,\y)=J(\y,\x),$$ 
hence $J_\ep(\x,\y)=J_\ep(\y,\x).$
Estimates on the heat kernel in bounded $C^{2}$ domains
  in $\bR^n$, \cite{D}, imply the following estimate on $J$ for $\x,\y\in \Omega$, see (\cite{AD},  also see \cite{CS}),
\beq 
 J(\x,\y)\asymp  \min \left (\frac{\delta(\x)\delta(\y)}{|\x-\y|^2}, 1\right ) \frac{1}{|\x-\y|^{N+2\al}}.\label{estp}
\eeq
Note that the blowup of $J$ at $\x=\y$ is the reason the principle value formulation above is needed.

Let 
$$
W(\x)=PV \int_\Omega (u(\x)-u(\y))J(\x,\y)d\y .
$$
To complete the proof of the theoreem, it suffices to prove the following three properties holding for all 
$u,w\in H^{3+N/2}(\Omega )$: 
$$
\noindent{\bf P1: }
\hspace{.5in}
W\mbox{ is in } L^1(\Omega) ,
$$
$$
{\bf P2: }\ 
u\mapsto \delta^{2\al +\tilde{\ep}}W \mbox{ is bounded from  } H^{3+N/2}(\Omega ) \mbox{ to }C(\overline{\Omega }),
$$
$${\bf P3:}\ \ \int_\Omega w(\x)PV\int_\Omega (u(\x)-u(\y))J(\x,\y)d\y d\x
$$
$$
=\int_\Omega u(\x)PV\int_\Omega (w(\x)-w(\y))J(\x,\y)d\y d\x .
$$
In fact,it is easy to see that parts i), resp. ii), of the theorem follow from {\bf P1}, resp. {\bf P2}, and part iii) follows from {\bf P3}.

 As a first step to proving {\bf P3}, observe that  since $J_\ep(\x,\y)$ is symmetric and bounded for fixed $\ep,$  we have 
\beq 
\int_\Omega w(\x)\int_\Omega (u(\x)-u(\y))J_\ep(\x,\y)d\y d\x
= \int_\Omega u(\x)\int_\Omega (w(\x)-w(\y))J_\ep(\x,\y)d\y d\x .\label{symm2}
\eeq
Thus, to prove {\bf P3}, it suffices to prove 
$$
 \lim_{\ep \to 0}\int_\Omega w(\x)\int_\Omega (u(\x)-u(\y))J_\ep(\x,\y)d\y d\x$$
 \beq =\int_\Omega w(\x) \lim_{\ep \to 0}\int_\Omega (u(\x)-u(\y))J_\ep(\x,\y)d\y d\x .
\label{lim}
\eeq
By Taylor's Theorem and \rq{estp},
\begin{eqnarray}
|(u(\x)-u(\y))\ J(\x,\y)|& \leq &
\sum_{|\beta|=1}|D^\beta u(\x+t(\x,\y )(\y-\x))| J(\x,\y)\nonumber \\
& \leq &C\| u\|_{C^2(\Omega )}\frac{1}{| \x-\y |^{N-1+2\al}}.\label{L1}
\end{eqnarray}
Since $1-2\al >0$, the last function is integrable in $\y$, uniformly in $\x$. Thus \rq{lim}  holds by Dominated Convergence. It is easy to see from \rq{L1} that ${\bf P1}, {\bf P2}$ also hold, and parts B,C of the theorem follow immediately. 
$\Box$ 

We remark that  in \cite{E}, we refined our estimate on $p$ for rectangles to extend the theorem to $\al <1$.

\subsection{Well-posedness for $\al <1$.}
In this section, we will discuss the well-posedness of \rq{beam2o}-\rq{init2o} in the case $\Omega \subset \bR^N$ is a bounded, $C^2$ domain for $\al <1$; the case of rectangles was addressed in \cite{E}.

Recall $A=-\Delta$ with operator domain 
$H^2(\Omega )\cap H^1_0(\Omega )$. Denote the eigenvalues and corresponding normalized eigenfunctions of $A$ by 
$\{\om_n, \f_n:\ n\in \bN \}$. Recall the Weyl asymptotics:
\beq
\om_n \asymp n^{2/N}.\label{weyl}
\eeq

We wish to discuss the well-posedness of the system
\begin{eqnarray}
u_{tt}+A^2 u+\rho A^\al u_t & = &0, \mbox{ on } Q,\label{beam2z}\\
u|_{ \Sigma_2}=\Delta u|_{ \Sigma_2}& = & 0\\
u|_{ \Sigma_1 }& =& f\\
\Delta u|_{ \Sigma_1 }& =& g\\
u(x,0)=u^0(x),\ u_t(x,0)& = & u^1(x), \ x\in \Omega.\label{init2z}
\end{eqnarray}
To this end, we consider the adjoint problem to the system \rq{beam2o}-\rq{init2o}:
\begin{eqnarray}
w_{tt}+A^2 w-\rho A^\al w_t & = & 0,\mbox{ on } Q,\label{w_prime'}\\
w|_{\Sigma}=\Delta w|_{\Sigma }& = & 0\\
w(x,T) & =& w^0(x)\\
w_t(x,T) & = & w^1(x)\label{init'}
\end{eqnarray}
with the observations $\pa_{\eta}\Delta w$ for Dirichlet control, and $\pa_{\eta} w$ for moment control.

We use Fourier series to represent $w$. Let $x\in \Omega$.
Set $w(x,t)=\sum_{n=1}^\infty a_n(t)\f_{n}(x)$. Then 
\rq{w_prime'} implies
$$
\sum (a_n''-\rho a_n'\om_{n}^{\al}+a_n\om_{n}^2  )\f_{n}(x)=0,
$$
hence
\beq
a_n''-\rho a_n'\om_{n}^{\al}+n_p\om_{n}^2=0,\ \ \forall n\in \bN.\label{an}
\eeq
Solving $\la^2-\rho \om_{n}^{\al}\la+\om_{n}^{2}=0$, we get
$$
\la=
\frac{\rho \om_{n}^{\al}\pm \sqrt{\rho^2\om_{n}^{2\al}-4\om_{n}^{2}}}{2}
$$
In what follows, it will be convenient to set 
\beq
\la_{n}^{\pm}=
\frac{-\rho \om_{n}^{\al}\pm \sqrt{\rho^2\om_{n}^{2\al}-4\om_{n}^{2}}}{2}.\label{lapm} 
\eeq
Thus 
\beq
w(x,t)=\sum_{n} (c_{n}^+e^{\la_{n}^+ (T-t)}+c_{n}^-e^{\la_{n}^- (T-t)})\f_{n}(x),\label{du}
\eeq
with coefficients $c_n^\pm$ determined by the terminal conditions.
The coefficients satisfy
\begin{eqnarray*}
c_{n}^+ + c_{n}^-=&w_{n}^0,\\
-\la_{n}^+c_{n}^+ - \la_n^-c_n^-=&w_n^1,
\end{eqnarray*}
where $ w_{n}^0$ and  $w_{n}^1$ are the Fourier coefficients of $ w_0$ and $w_1$ with respect to $\f_{n}$.
This gives the following expression
\begin{eqnarray}
c_{n}^+=&-(w_{n}^1+\la_{n}^-w_{n}^0)/\sqrt{\rho^2\om_{n}^{2\al} -4\om_{n}^2},\nonumber \\
c_{n}^-=&(w_{n}^1+\la_{n}^+w_{n}^0)/\sqrt{\rho^2\om_{n}^{2\al}-4\om_{n}^2}.\label{cn'}
\end{eqnarray}

\begin{lemma}\label{wpadj}
Let $j=0$ or $j=2$.
Given $(w^0,w^1)\in X^{3-j}\times X^{1-j}$, there exists a unique solution to \rq{w_prime'}-\rq{init'}, such that
$w\in C^0(0,T;X^{3-j})\cap C^1(0,T;X^{1-j}).$ 
Furthermore, there exists a constant $C>0$ independent of $u^0,u^1$ such that, for  $j=0$, we have  
$$\int_{\Sigma_1} |\pa_\eta  \Delta w |^2\leq
C(\| w^0\|^2_{X^3}+\| w^1\|^2_{X^1}),
$$ 
and for $j=2$ we have 
$$\int_{\Sigma_1} |\pa_\eta   w |^2\leq
C(\| w^0\|^2_{X^1}+\| w^1\|^2_{X^{-1}}).
$$
\end{lemma}
The proof of this result is a standard exercise in Fourier series. We mention only a couple of key points, leaving the details to the reader. The assumptions $\al \leq 1, \rho <2$ and the Weyl asymptotics 
 $\om_n\asymp n^{2/N}$ imply
 $$
\Re (\la_n^\pm )=-\rho \om_n^{\al }/2\asymp -n^{2\al /N}/2.\ 
\Box $$

Now suppose $u,w$ are sufficiently regular to permit integration by parts.
 Then we derive the following weak formulation of $u$ to be a solution to \rq{beam2o}-\rq{init2o}. Multiplying 
\rq{w_prime'} by $u$ and integrating by parts, we get 
\begin{eqnarray*}
0&=&[\int_{\Omega} (w_t(x,t)u(x,t)-\rho A^\al  w(x,t)u(x,t)) \ dx]_0^T\nonumber \\
& & -[\int_{\Omega} w(x,t),u_t(x,t)\ dx]_0^T
-
\int_{\Sigma_1}
(f \pa_\eta  \Delta w -g \pa_\eta  w)
\ d\sigma dt,
\end{eqnarray*}
where $\sigma $ is the volume element on $\pa \Omega$ induced by $dx$. Setting $g=0$, we define $u$ to be a weak solution if for all $(u^0,u^1)\in X^3\times X^1$, 
\begin{eqnarray}
0&=&[\int_\Omega w_t(*,t),u(*,t)\>_{X^1,X^{-1}}-\rho \<A^\al  w(*,t),u(*,t) \>_{X^1,X^{-1}}]_0^T\nonumber \\
& & -[\< w(*,t),u_t(*,t)\>_{X^3,X^{-3}}]_0^T-
\int_{\Sigma_1}
f \pa_\eta  \Delta w 
\ d\sigma dt.
\label{weak1b}
\end{eqnarray}
Setting $f=0$, we define $u$ to be a weak solution if for all $(u^0,u^1)\in X^1\times X^{-1}$, 
\begin{eqnarray}
0&=&[\< w_t(*,t),u(*,t)\>_{X^{-1}1,X^{1}}-\rho \<A^\al  w(*,t),u(*,*,t) \>_{X^{-1}1,X^{1}}]_0^T\nonumber \\
& & -[\< w(*,t),u_t(*,t)\>_{X^1,X^{-1}}]_0^T+
\int_{\Sigma_1}
 g \pa_\eta  w
\ d\sigma dt.
\label{weak1c}
\end{eqnarray}
\ 

{\em Proof of Corollary \ref{wp<1}}

This is a  standard duality argument using Lemma \ref{wpadj} and \rq{weak1b}, \rq{weak1c}.
The interested reader is referred to \cite{E}, where the same argument is used when $\Omega$ is a rectangle and $\al <1$. 

\subsection{Spectrum on Rectangles}\label{2.3}

In what follows, we write a rectangle in the form
$$
R=(0,\pi)\times \tilde{R},
$$
with standard coordinates labelled $(x,y_1,\ldots ,\y_{N-1}) $. Also, we label $y=(y_1,....y_{N-1})$.

Consider an eigenvalue problem on $(0,\pi):$
$$
- \psi ''  = \la \psi ,\
\psi (0) =\psi (\pi)= 0.\\
$$
Clearly an orthonormal  basis of eigenfunctions of this problem is $\{ \psi_n; \ n\in\bN\} $, 
$\psi_n(x)=\sqrt{\frac{2}{\pi}}\sin ( n x)$, with corresponding eigenvalues 
$n^2.$

 Recall the eigenvalues of $\Delta_{\tilde{R}}$, listed in non-decreasing order,  are denoted $\{ -\kappa_m,m\in \bN \}$, with corresponding normalized eigenfunctions $\phi_m(y)$, and hence the eigenvalues of $A$ are 
 \beq 
 \om_{m,n}:=n^2+\kappa_m,\ \kappa_m\sim m^{2/(N-1)} \label{om0}
 \eeq 
 with corresponding normalized eigenfunctions $\f_{m,n}(x,y)=\psi_n(x)\phi_m(y)$. 

Then, with this notation, we have
\beq 
\la_{m,n}^{\pm}=
\frac{-\rho \om_{m,n}^{\al}\pm \sqrt{\rho^2\om_{m,n}^{2\al}-4\om_{m,n}^{2}}}{2},\label{lapma} 
\eeq
and 
the solution $w$ of the adjoint problem \rq{w_prime'}-\rq{init'} will be 
\beq
w(x,y,t)=\sum_{m,n} (c_{m,n}^+e^{\la_{m,n}^+ (T-t)}+c_{m,n}^-e^{\la_{m,n}^- (T-t)})\f_{m,n}(x,y),\label{du9}
\eeq
with 
\begin{eqnarray}
c_{m,n}^+=&-(w_{m,n}^1+\la_{m,n}^-w_{m,n}^0)/\sqrt{\rho^2\om_{m,n}^{2\al} -4\om_{m,n}^2},\nonumber \\
c_{m,n}^-=&( w_{m,n}^1+\la_{m,n}^+w_{m,n}^0)/\sqrt{\rho^2\om_{m,n}^{2\al}-4\om_{m,n}^2},\label{cnA}
\end{eqnarray}
where $ w_{m,n}^0$ and  $w_{m,n}^1$ are the Fourier coefficients of $ w^0$ and $w^1$ with respect to $\f_{m,n}$.

Set $\bK:=\{\pm 1, \pm 2, \ldots \}$.
In what follows we will use the \textit{spectrum}
\beq\label{La}
\Lambda=\{\la_{m,k},m\in \bN, k\in\bK\}, \ 
		\la_{m,k} =\left \{
		\begin{array}{cc}
			-i\la_{m,k}^+, & k>0,\\
			-i\la_{m,|k|}^-, & k<0.
		\end{array}
		\right .
\eeq	
We then extend other terms above: $\f_{m,k}:=\f_{m,|k|}$, and 
\begin{eqnarray}\label{cn1}
c_{m,k}=&(w_{m,k}^1-i\la_{m,k}w_{m,k}^0)/\sqrt{\rho^2\om_{m,k}^{2\al} -4\om_{m,k}^2},\ k>0,\\
c_{m,k}=&(-w_{m,|k|}^1+i\la_{m,|k|}w_{m,|k|}^0)/\sqrt{\rho^2\om_{m,|k|}^{2\al}-4\om_{m,|k|}^2},\ k<0.
\end{eqnarray}
Thus 
\beq
w(x,y,t)=\sum_{m\in \bN,k\in \bK} c_{m,k}e^{i\la_{m,k} (T-t)}\f_{m,k}(x,y).\label{du2}
\eeq
We now examine the gap properties of 
$\{ \la_{m,k}\}$. 
In this paper, we assume $\rho <2$, $\al \leq 1$. It follows immediately from \rq{om0} and \rq{lapma} that for each fixed $m\in \bN$, there exists $\gamma>0$ independent of $m$ such that 
\beq
\inf \{ |\la_{m,k_1}-\la_{m,k_2}|\geq \gamma, \forall k_1,k_2\in \bK.\label{gap1}
\eeq 
Such gap conditions might occur for other $\al,\rho$, provided $\al \leq 3/2$, but verifying this would be a non-trivial number theory problem and beyond the scope of this paper. For further discussion on this issue, the interested reader is referred to \cite{AEI}, \cite{AE}.


\subsection{Complex window estimate}\label{wins}

An important ingredient in our analysis will be a complex window estimate for  $\{ e^{i\la_{m,k}t}:m\in \bN,\ k\in \bK\}$ proven 
 \cite{AIS}, and then extended in \cite{B}. It will convenient to express our results for the modified frequency set,
 $\La^m =\{ \la_{m,k}: k\in \bK\}$, with $\la_{m,k}$ defined in \rq{La}, so that $\La^m\in \bC^+$ for each $m$.

We introduce a function $\nu^m(s)$,
which describes the density of $\Lambda^m$ 
$$ 
\#\{ \la_{m,j}\in \Lambda\setminus \la_{m,k}  :  |\la_{m,j} -\la_{m,k}| < s\} \leq \nu^m (s), \forall k.
$$
By \rq{om0},\rq{lapma}, and \rq{gap1},  there exists $R_0$ independent of $m$ such that 
$\Lambda^m$ satisfies 
\beq
\nu^m(s)=0, \ s<R_0,\label{separ}
\eeq
 for a positive $R_0$. This assumption is equivalent to the
separability of $\Lambda^m$ being uniform in $m$.

We now estimate $\nu^m$. By \rq{om0}, for fixed $m$, 
$\om_{m,j}-\om_{m,k}=j^2-k^2.$ Also, for $\al\leq 1$ and $\rho <2$, \rq{lapma} shows that 
 there exist positive constants, $C_1,C_2$ depending on $\rho,\al$ but not on $m$, such that one can choose $\nu (r)$ satisfying
\beq \label{nu}
C_1r^{1/2}\leq \nu^m(r)\leq C_2r^{1/2},\ \forall r>2R_0.
\eeq 
Thus applying \cite{AIS}, we obtain the following result:
\begin{thm}\label{win}
Let $\al \leq 1$ and $\rho <2$.   Fix $m$.

A) 
For any $T,T'>0$, the operator 
$$C_{T,T'}: \sum_k  c_{k} e^{i\la_{m,k} t} \mapsto \{c_{k}e^{-T'\Im(\la_{m,k})};k\in \bK\}$$ 
is bounded from
$L^2(0,T)$ to $l^2$ with
$$
        \| C_{T,T'}\| \leq C\,\Psi (T ,T'),
$$
where $\Psi
(T,T'):=  e^{2Q(T)} e^{1/T'}$ with $Q(T)$ a constant that depends on $T$, and $C$ depends of $\al, \rho$. In particular, $ \| C_{T,T'}\|$ will be independent of $m$.

B) Then there exists a family of functions $\{  g_{m,n}(t); n\in \bK  \} $ in $L^2(0,T)$ satisfying 
$$ \int_0^T \overline{g_{m,n}}(t)e^{i\la_{m,k}t}\ dt =\delta_{n,k}.$$

Furthermore, there exist positive constants $C_2,C_3$ depending  only on $R_0,T,C_1,C_2$ such that
\beq
\| g_{m,n}\|_{L^2(0,T)}\leq C_2 e^{C_3(\Im (\la_{m,n})^{1/2}}.\label{boe}
\eeq
\end{thm}

\section{Observability on rectangles}

\subsection{Proof of Theorem \ref{prod}, part A}\label{3.1}
Letting $T>0$, recall the solution of the adjoint problems is 
\beq
w(x,y,t)=\sum_{m,k} c_{m,k}e^{i\la_{m,k} (T-t)}\f_{m,k}(x,y),\label{du3}
\eeq
with 
\begin{eqnarray}\label{cn1b}
c_{m,k}=&(w_{m,k}^1-i\la_{m,-k}w_{m,k}^0)/\sqrt{\rho^2\om_{m,k}^{2\al} -4\om_{m,k}^2},\ k>0\\
c_{m,k}=&(-w_{m,|k|}^1+i\la_{m,|k|}w_{m,|k|}^0)/\sqrt{\rho^2\om_{m,k}^{2\al}-4\om_{m,k}^2},\ k<0.
\end{eqnarray}
Thus there exists $C$ independent of $m,k$ such that 
$$
|c_{m,k}|^2\leq C(|w^0_{m,k}|^2+|w^1_{m,k}|^2/\om_{m,k}),\ \forall m,k.
$$

We will assume, without loss of generality, 
that $\Gamma_1$ is a rectangular subset of the rectangle $R$'s face $\{ x=\pi\}\times \tilde{R}$, with edges parallel to the coordinate axis as in \rq{rr}.
 Recalling \rq{weak1b},   $u$ is a solution to \rq{beam2o}-\rq{init2o} if, for all $(w^0,w^1)\in X^3\times X^1$,
\begin{eqnarray}
0&=&[\< w_t(*,*,t),u(*,t)\>_{X^1,X^{-1}}-\rho \<A^\al  w(*,*,t),u(*,*,t) \>_{X^1,X^{-1}}]_0^T\nonumber \\
& & -[\< w(*,*,t),u_t(*,*,t)\>_{X^3,X^{-3}}]_0^T\nonumber\\
&&-
\int_0^T\int_{\Gamma_1}
\big ((\pa^3_{x} +\pa_{x} \Delta_{\tilde{R}})w \big  ) (\pi,y,t)
f(y,t)dydt.
\label{weak1a}
\end{eqnarray}
It is well known by duality that Theorem \ref{prod}, part A, is equivalent to the following observability estimate:
\begin{prop}\label{thm7}
Let $\Gamma_1\subset \tilde{R}$ be a rectangle with faces parallel to the coordinate axes. Let $T>0$.
The following estimate holds:
$$
\int_0^T\int_{\Gamma_1}
\big ((\pa^3_{x} +\pa_{x} \Delta_{\tilde{R}})w \big  ) (\pi,y,t)
f(y,t)dydt\geq C(\| w(*,*,0)\|_{X^3}^2+\| w_t(*,*,0)\|^2_{X^1},
$$
with $C$ independent of $w^0,w^1$.
\end{prop}
Proof: By \rq{du3},
\begin{eqnarray}
\| w(*,*,0)\|_{X^3}^2+\| w_t(*,*,0)\|^2_{X^1}& = & \sum_{m,k}|c_{m,k}e^{i\la_{m,k}T}|^2
({\om_{m,k}}^{3/2}+(\om_{m,k})^{1/2}|\la_{m,k}|)\nonumber \\
& = & \sum_{m,k}|c_{m,k}|^2e^{-\rho \om_{m,k}^\al T}({\om_{m,k}}^{3/2}+(\om_{m,k})^{1/2}|\la_{m,k}|).\label{e0}
\end{eqnarray}
On the other hand, we set $\Gamma_1=S$ in Proposition \ref{UM}, and then apply Theorem \ref{win} with $T'=T/2$:
$$
\int_0^T\int_{\Gamma_1}
|\big (\pa^3_{x} +\pa_{x} \Delta_{\tilde{R}} \big  ) w(\pi,y,t)|^2dydt
$$
\begin{eqnarray}& = & \int_0^T\int_{S}
|\sum_{m,k}(-1)^{k+1}|k|(k^2+\kappa_m)c_{m,k}e^{i\la_{m,k}(T-t)}\phi_m(y)|^2dydt\nonumber \\
& \geq  & C\sum_m\int_0^T
|\sum_{k}(-1)^{k+1}m^{-2}|k|(k^2+\kappa_m)c_{m,k}e^{i\la_{m,k}(T-t)}|^2dt\nonumber  \\
& \geq  & C_{\delta}\sum_m
\sum_{k}\big (m^{-2}|k|(k^2+\kappa_m)\big )^2|c_{m,k}|^2e^{- \rho \om_{m,k}^\al T/2}.\label{ob}
\end{eqnarray}
By \rq{lapma}, \rq{om0}, there exists a constant $C>0$ such that
$$
m^{-2}|k|(k^2+\kappa_m)e^{- \rho \om_{m,k}^\al T/2}\geq C  \sqrt{\om_{m,k}}(\om_{m,k}+|\la_{m,k}|)e^{-\rho \om_{m,k}^\al T}, \forall m,k.
$$
From this, \rq{e0} and \rq{ob}, the proposition follows.$\Box$

\subsection{Proof of Theorem \ref{prod}, part B}
For this section, we have $\Gamma_1\subset \{ \pi \} \times \tilde{R}$, same as in Section \ref{3.1}. Recall $\Sigma_j=\Gamma_j\times (0,T)$. Recall we use coordinates on $R$: $(x,y_1,\ldots ,y_{N-1})$ with $x\in (0,\pi)$.
Consider 
\begin{eqnarray}
u_{tt}+\Delta^2 u+\rho A^\al u_t& = & 0 \mbox{ on }Q, \label{c1}\\
u|_{\Sigma_2}=\Delta u|
_{\Sigma_2} & = & 0,\ t>0,\label{bcr11}\\
\Delta u|_{\Sigma_1} & = & g(y,t),\ t>0,\label{bcr21}\\ 
 u|_{\Sigma_1} & = & 0,\ t>0,\label{bcr31}\\
u(x,y,0)=0,\ u_t(x,y,0)& = & 0.\label{c5}
\end{eqnarray}
By \rq{weak1c},  $u$ is a solution to \rq{c1}-\rq{c5} if for all $(w^0,w^1)\in X^1\times X^{-1}$,
\begin{eqnarray}
0&=&\< w_t(*,T),u(*,T)\>_{X^{-1},X^{1}}-\<\A w(*,T),u(*,T) \>_{X^{-1},X^{1}}\nonumber \\
& & -\< w(*,T),u_t(*,T)\>_{X^1,X^{-1}}+ \int_{\Sigma_1} \pa_{x}w(\pi ,y,t)g(y,t)\ dydt.
 \label{weak1}
\end{eqnarray}
 Theorem \ref{prod}, part B,  now follows from the following observability estimate.
\begin{prop}\label{thm8}
Let $T>0$.
Let $\Gamma_1\subset \tilde{R}$ be a rectangle with faces parallel to the coordinate axes. 
The following estimate holds:
$$
\int_{\Sigma_1}
| \pa_{x}  w(\pi,y,t)|^2dydt\geq C(\| w(*,*,0)\|_{X^1}^2+\| w_t(*,*,0)\|^2_{X^{-1}}),
$$
with $C$ independent of $w^0,w^1$.
\end{prop}
The proof of this result mimics the proof of Proposition \ref{thm7}, and is left to the reader. 
 
\section{Proof of Theorem \ref{prod3}}

Let $(r,\theta)$ be the standard polar coordinates on $D$, so the standard Lebesgue measure is given by $rdrd\theta$. We denote $\Gamma =S^1=\{ r=1\}$.
Define an orthonormal basis of $L^2(S^1)$ by 
$$
\phi_n(\theta )=\left \{
\begin{array}{cc}
1/\sqrt{2\pi} & n=0,\\
\cos{n\theta}/\sqrt{\pi} & n>0,\\
\sin{n\theta}/\sqrt{\pi} & n<0.
\end{array}
\right .
$$
\begin{prop}\label{UMa}
Let $(a,b)$ be a subinterval of $S^1$. 
Let $h(n)=n^2$ for $n\neq 0$, and $h(0)=1$.
Then for any  $\ell^2$ sequence $\{ c_n: n\in \bZ\} $, 
$$
\int_a^b \left|\sum_{n\in \bZ}c_n\phi_n(x)\right|^2\ dx\geq 
C\sum_{n\in \bZ}|c_n|^2/h(n)^2,
$$
where the constant $C>0$ depends on $(a,b)$.
\end{prop}
The proof of this result is deferred until the appendix. 

The normalized eigenfunctions for the Dirichlet Laplacian on $D$, with their corresponding eigenvalues, are well known to be 
$$
\f_{m,n}(r,\theta)=\frac{\sqrt{2}}{|J_{n+1}(\beta_{m,n})|}J_n(\beta_{m,n}r)\phi_n(\theta ),\ \om_{m,n}= \beta_{m,n}^2   +n^2,\ m\in \bN, \ n\in \bZ .
$$
where for each $n$,  $\beta_{m,n}>0$ are the solutions, listed in increasing order, of $J_n(\beta )=0$, and 
$J_n$ are the bounded solutions of the Bessel equations 
$$
r^2x''(r)+rx'(r)+(\la r^2-n^2)=0, \ r(1)=0.
$$

Letting $T>0$, consider the adjoint to the system 
\rq{beam2o}-\rq{init2o}:
\begin
{eqnarray}
w_{tt}+A^2 w-\rho A^\al w_t & = & 0,\label{wD}\\
w|_{S^1}=\Delta w|_{S^1}& = & 0\\
w(x,T) & =& w^0(x)\\
w_t(x,T) & = & w^1(x).\label{init_prime}
\end{eqnarray}
Arguing as in previous sections, 
\rq{wD} implies
\beq
w(r,\theta ,t)=\sum_{k\in \bK,n\in \bZ} c_{k,n}e^{i\la_{k,n} (T-t)}\f_{k,n}(r,\theta),\label{du2d}
\eeq
where 
\beq\label{LaD}
\Lambda=\{\la_{k,n},k\in \bK, n\in\bZ\},\ 
		\la_{k,n} =\left \{
		\begin{array}{cc}
			-i\la_{k,n}^+, & k> 0,\\
			-i\la_{|k|,n}^-, & k<0,
		\end{array}
		\right .
\eeq
\beq 
\la_{m,n}^{\pm}=\frac{-\rho \om_{m,n}^{\al}\pm \sqrt{\rho^2\om_{m,n}^{2\al}-4\om_{m,n}^{2}}}{2}, 
\eeq
$\f_{k,n}:=\f_{|k|,n}$, $\om_{k,n}=\om_{|k|,n}$, and 
\begin{eqnarray}\label{cn1c}
c_{k,n}=&(w_{k,n}^1-i\la_{k,-n}w_{k,n}^0)/\sqrt{\rho^2\om_{k,n}^{2\al} -4\om_{k,n}^2},\ k>0 ,\\
c_{k,n}=&(-w_{k,n}^1+i\la_{|k|,n}w_{k,n}^0)/\sqrt{\rho^2\om_{k,n}^{2\al}-4\om_{k,n}^2},\ k<0.
\end{eqnarray}

Applying a standard duality argument to \rq{weak1b}, we see that 
proving Theorem \ref{prod3} is equivalent to the following observability estimate.
\begin{prop}\label{thm9}
Let $\Gamma_1$ be an interval in $S^1$. Let $T>0$.
The following estimate holds:
$$
\int_0^T\int_{\Gamma_1}
| (\pa_r \Delta  w)(1,\theta ,t)|^2\ d\theta dt\geq C(\| w(*,*,0)\|_{X^3}^2+\| w_t(*,*,0)\|^2_{X^1}),
$$
with $C$ independent of $w^0,w^1$.
\end{prop}
Proof: 
By \rq{du2d},
\begin{eqnarray*}
(\Delta w_r(1,\theta ,t))& =&\sum_{k\in \bK ,n\in \bZ} \om_{k,n}c_{k,n}e^{i\la_{k,n} (T-t)}\pa_r\f_{k,n}(1,\theta)\\
& =  & \sum_{k\in \bK,n\in \bZ} \om_{k,n}c_{k,n}e^{i\la_{k,n} (T-t)}\frac{\sqrt{2}\beta_{k,n}J_n'(\beta_{k,n})}{|J_{n+1}(\beta_{k,n})|}\phi_n(\theta).
\end{eqnarray*}
We begin to estimate $\int \int|\frac{\pa \Delta w}{\pa r} (1,\theta, t)|^2$. By Proposition \ref{UMa}, we have
$$
\int_{\Gamma_1}|\sum_n\sum_k\om_{k,n}c_{k,n}e^{i\la_{k,n}(T-t)}\frac{\sqrt{2}\beta_{k,n}J_n'(\beta_{k,n})}{|J_{n+1}(\beta_{k,n})|}\phi_n(\theta)|^2d\theta$$
$$\geq\sum_n\frac{1}{h(n)^2}|\sum_k\om_{k,n}c_{k,n}e^{i\la_{k,n}(T-t)}\frac{\sqrt{2}\beta_{k,n}J_n'(\beta_{k,n})}{|J_{n+1}(\beta_{k,n})|}|^2.$$
We now list some
 properties of Bessel functions, see  \cite{K}. For any $n$,
\begin{itemize}
\item  The terms $\{ \beta_{m,n}\}$ are
simple, and form a strictly increasing sequence tending to infinity.

\item The difference sequence $\beta_{m+1,n}-\beta_{m,n}$ converges to $\pi$.

\item The difference sequence $\beta_{m+1,n}-\beta_{m,n}$ is decreasing. 

\item  $J_n'(\beta_{m,n})=-J_{n+1}'(\beta_{m,n}).$

\end{itemize}
Fix $n$. Then
since $\rho < 2$ and $\al \leq 1/2$, we have 
$\sqrt{\rho^2\om_{k,n}^{2\al}-4\om_{k,n}^2}$
is purely imaginary. It is then easy to see that 
$\{\la_{k,n}: k\in \bK\} \subset \bC^+$ is simple, and there exists a constant $\gamma >0$, independent of $n$, such that the gap condition is satisfied:
$$
\inf_{j\neq k}|\la_{j,n}-\la_{k,n}|>\gamma .
$$
Hence by Theorem \ref{win}, where we set $T'=T/2$,  there exists a constant $C$ such that 
$$
\int_0^T
|\sum_k\om_{k,n}c_{k,n}e^{i\la_{k,n}(T-t)}\frac{\sqrt{2}\beta_{k,n}J_n'(\beta_{k,n})}{|J_{n+1}(\beta_{k,n})|}|^2\ dt
\geq 
C 
\sum_k |\om_{k,n}c_{k,n}\beta_{k,n}|^2e^{-\rho \om_{k,n}^\al T/2}.
$$
Combining, we get 
\beq 
\int_{\Sigma_1}|\frac{\pa \Delta w}{\pa r}(1,\theta, t)|^2\ d\theta dt\geq C 
\sum_n \frac{1}{h(n)^2}\sum_k |\om_{k,n}c_{k,n}\beta_{k,n}|^2e^{-\rho \om_{k,n}^\al T/2}.\label{lb}
\eeq
On the other hand, 
\begin{eqnarray}
\| w(*,*,0)\|^2_{X^3}+\|w_t(*,*,0)\|^2_{X^1}&=&
\sum_{m\in \bN}\sum_{n \in \bZ}
|c_{m,n}|^2(\om_{m,n}^{3/2}+\om_{m,n}^{1/2}|\la_{m,n}|)e^{-\rho\om_{m,n}^\al T}\nonumber \\
&=&\frac{1}{2}
\sum_{k\in \bK}\sum_{n \in \bZ}
|c_{k,n}|^2(\om_{k,n}^{3/2}+\om_{k,n}^{1/2}|\la_{k,n}|)e^{-\rho\om_{k,n}^\al T}. \label{ini}
\end{eqnarray}
Then
combining \rq{lb} and \rq{ini}, the observability estimate follows. $\Box$

\section{Well-posedness and controllability in classical case: $\al =1$.}
We assume 
	$\al =1 $ and $\rho <2.$ Let $N\geq 2$.
    
In this section, we adopt the notation of Section \ref{3.1}. Thus 
    we consider $\Omega=(0,\pi)\times M$, with $M\subset \bR^{N-1}$ a bounded $C^2$ domain,
with standard coordinates labelled $(x,y_1,\ldots ,y_{N-1}) $. Also, we label $y=(y_1,....y_{N-1})$.   
    It will be sometimes convenient to write the Laplacian for $\Omega$ is 
    $\Delta =\pa_x^2+\Delta_y$. Let  $\Gamma_1$ be a relatively open subset of the face $\{ x=\pi\}$, and $\Gamma_2$  the interior of $\Gamma\setminus \Gamma_1$.

We discuss null-controllability for 
\begin{eqnarray}
u_{tt}+A^2 u+\rho A u_t& = & 0,\label{beamd1m}\\
u(*,*,t) & = & 0, \mbox{ on }\Sigma_2,\label{bc1dm}\\
u(\pi,y,t) & = & f(t,y)\mbox{ on }\Sigma_1,\label{bc2dm}\\ 
\Delta u& = & 0\mbox{ on }\Sigma,\label{bc3dm}\\
u(x,y,0)=u^0(x,y),\ u_t(x,y,0)& = & u^1(x,y).\label{initdm}
\end{eqnarray}
For this, we will use the moment method.

Consider an eigenvalue problem
$$
- \psi ''  = \la \psi ,\
\psi (0) =\psi (\pi)= 0.\\
$$
Denote an orthonormal  basis of eigenfunctions by $\{ \psi_n; \ n\in\bN\} $, 
$\psi_n(x)=\sqrt{\frac{2}{\pi}}\sin ( n x)$, with corresponding eigenvalues 
$n^2.$
The eigenvalues of $\Delta_{M}$ are denoted $\{ -\kappa_m,m\in \bN \}$, with corresponding normalized eigenfunctions $\phi_m(y)$, and hence the eigenvalues of $A$ are 
 \beq 
 \om_{m,n}^2:=n^2+\kappa_m,\ \kappa_m\sim m^{2/(N-1)} \label{om}
 \eeq 
 with corresponding normalized eigenfunctions $\f_{m,n}(x,y)=\psi_n(x)\phi_m(y)$. 

Fix $t$. To discuss well-posedness, we first
 find $Z=Z(x,y,t)$ solving
$$
\Delta Z=0 \mbox{ on }\Omega, \ Z|_{\Gamma_1}=f(\pi ,y,t),\ Z|_{\Gamma_{2}}=0.
$$
We write $f(\pi ,y, t)=\sum_m f_m(t)\phi_m(y).$
By a standard Fourier series argument, we have
$$
Z(x,y,t)=\sum_m \frac{f_m(t)}{\sinh (\kappa_m \pi )}\sinh (\kappa_mx)\phi_m(y).
$$

\begin{prop}\label{wpd}
Let $T>0$ and $f\in H_*^2(0,T;L^2(M))$. 
Let $u$ solve our Dirichlet control system \rq{beamd1m}-\rq{initdm}. 
Suppose $(u^0,u^1)\in X^3\times X^1$. 
Then \
$$
u^f(x,y,t)=Z(x,y,t) +v(x,y,t),
$$
with 
$$
v\in C(0,T;X^3)\cap C^1(0,T;X^1).
$$
Thus 
$$
u\in C^1(0,T;X^0).
$$
\end{prop}
Proof:
We assume for the moment that $f\in C_*^2((0,T)\times M)$.
Set $v=u-Z.$
	Then $v$ satisfies
	\begin{eqnarray}
		v_{tt}+\Delta^2 v-\rho\Delta v_t & = & -Z_{tt},\label{v1p}\\
		v|_{\Gamma}=\Delta v|_{\Gamma}& = & 0\\
		v(x,y,0) & =& u^0(x,y)-Z(x,y,0)\\
		v_t(x,y,0) & = & u^1(x,y)-Z_t(x,y,0).\label{v4p}
	\end{eqnarray}
By hypothesis, we have 
	$$
	f_m(y,0)=f_m'(y,0)=0, \ \forall m ,
	$$
hence 
$$
Z(x,y,0)=Z_t(x,y,0)=0.
$$
    
	We will control $u$ by controlling $v+U.$ 
We further decompose our problem by setting $v=v^0+v^f$, where $v^0$ is the solution of 
	\rq{v1p}-\rq{v4p} with $f=0$, and $v^f$ is the solution with $u^0=u^1=0.$

We write $\sinh (\kappa_mx)=\sum_n \xi_{m,n}\psi_n(x)$.

	Set $v^f=\sum a_{m,n}(t)\f_{m,n}(x,y)$. Then 
		\rq{v1p} implies
		$$
		\sum_{m,n} (a_{m,n}''+\rho a_{m,n}'\om_{m,n}+a_{m,n}\om_{m,n}^2  )\f_{m,n}(x)=-\sum_{m,n} \frac{f_m''(t)}{\sinh (\kappa_m \pi )}\sinh (\kappa_mx)\phi_m(y).
		$$
		hence, for each $m$,
		$$
		\sum_n(a_{m,n}''+\rho \om_{m,n}a_{m,n}'+\om_{m,n}^2a_{m,n})\psi_n(x)=-\sum_{n} \frac{f_m''(t)}{\sinh (\kappa_{m} \pi )}\xi_{m,n}\psi_n(x),\ a_{m,n}(0)=a_{m,n}'(0)=0,
	$$
hence 
\beq
		a_{m,n}''+\rho \om_{m,n}a_{m,n}'+\om_{m,n}^2a_{m,n}=-\frac{f_m''(t)}{\sinh (\kappa_{m} \pi )}\xi_{m,n},\ a_{m,n}(0)=a_{m,n}'(0)=0,\ \forall m,n.\label{anp}
\eeq

		In what follows, it will be convenient to 
		set
		$$\beta_{m,n}=\frac{-\rho \om_{m,n}}{2}, \ 
		\al_{m,n}=\om_{m,n}\frac{\sqrt{4-\rho^2}}{2},\mbox{ and }\la_{m,n}^{\pm}=\beta_{m,n}\pm i\al_{m,n}.
		$$
	
		We next find $v^f(x,y,t)$ using variation of parameters in \rq{anp}.
		We have the  Wronskian equalling
		$-2i\al_{m,n}e^{-2\beta_{m,n}t}$.
		Hence, for all $m,n$, we have
        \begin{eqnarray}
			a_{m,n}(t)
			& =& \frac{\xi_{m,n}}{2i\al_{m,n}\sinh(\kappa_{m}\pi)}
			\int_0^t f_m''(s)\big (
			e^{\la_{m,n}^+(t-s)}-e^{\la_{m,n}^-(t-s)}\big )ds,
			 \label{vf2p} \\
            a_{m,n}'(t)
			& =& \frac{\xi_{m,n}}{2i\al_{m,n}\sinh(\kappa_{m}\pi)}
			\int_0^t f_m''(s)\big (
			(\la_{m,n}^+e^{\la_{m,n}^+(t-s)}-\la_{m,n}^-e^{\la_{m,n}^-(t-s)}\big )ds. \label{vfp}
		\end{eqnarray}		
 These formulae were derived assumed $f\in C_*^2$, but we now extend them by continuity to 
 $f\in H_*^2(0,T;L^2(M))$.

 We have, for fixed $m$,
\beq
\frac{|\xi_{m,n}|}{\sinh (\kappa_m \pi)}\asymp \frac{n}{n^2+\kappa_m^2}
,\ |\al_{m,n}| \asymp |\la_{m,n}|\asymp \om_{m,n}\asymp (n^2+\kappa_m).\label{est1}
\eeq
Hence by \rq{vf2p}, \rq{vfp}, we get
$$
v^f\in C(0,T;X^3)\cap C^1(0,T;X^1).
$$
It is easy to see that 
$$
v^0\in C(0,T;X^3)\cap C^1(0,T;X^1), \mbox{ and }Z \in C^1(0,T;X^0),
$$
so  Proposition \ref{wpd}  follows from 
$u^f=v^f+v^0+Z$.$\Box$

{\bf Proof of Theorem \ref{prod1}:}
	The null controllability is for \rq{beamd1m}-\rq{initdm} is equivalent to the existence of $f$ such that 
	\beq 
	v^f(x,y,T)=-v^0(x,y,T); \ v_t^f(x,y,T)=-v^0_t(x,y,T) .
	\label{cont2p}
	\eeq

To prove the theorem, we will find an associated moment problem. We need to express the ``free wave", $v^0$, as a Fourier series. Suppose for $j=0,1$, the initial conditions have Fourier expansion $u^j(x,y)=\sum_{m,n} u_{m,n}^j\f_{m,n}(x,y)$. Then
		from the above, 
		$$v^0(x,y,t)=\sum_{m,n}\big ( c^1_{m,n}e^{\la_{m,n}^+t}+c^2_{m,n}e^{\la_{m,n}^-t}\big )\f_{m,n}(x,y),$$
		with 
		$$c_{m,n}^2=\frac{\la_{m,n}^+u_{m,n}^0-u_{m,n}^1}{\la_{m,n}^+-\la_{m,n}^-},
		c^1_{m,n}=u^0_{m,n}-c^2_{m,n}.
		$$
We use the following notation: 
		\beq \label{tc1p}
		-v^0(x,y,T)=
		\sum_{m,n}\big (- c^1_{m,n} e^{\la_{m,n}^+T}
		-c_{m,n}^2e^{\la_{m,n}^-T}  \big )
		\f_{m,n}(x,y)=:\sum_{m,n}\chi_{m,n}^1\f_{m,n}(x,y),
		\eeq 
		\beq \label{tc2p}
		-v^0_t(x,y,T)=\sum_{m,n}\big (- c^1_{m,n}\la_{m,n}^+e^{\la_{m,n}^+T}
		-c_{m,n}^2\la_{m,n}^-e^{\la_{m,n}^-T}  \big )
		\f_{m,n}(x,y)=:\sum_{m,n}\chi_{m,n}^2\f_{m,n}(x,y).
		\eeq 
Thus 
\beq
\om_{m,n}^{3/2}|\chi_{m,n}^1|\leq Ce^{-\rho \om_{m,n}T/2};\ 
\om_{m,n}^{1/2}|\chi_{m,n}^2|\leq Ce^{-\rho \om_{m,n}T/2}.\label{chi}
\eeq
By \rq{cont2p}, \rq{vf2p},\rq{vfp}, \rq{tc1p},\rq{tc2p}, we get 
\begin{eqnarray*}
\chi_{m,n}^1 & =& \frac{\xi_{m,n}}{2i\al_{m,n}\sinh(\kappa_{m}\pi)}\int_0^T\big ( f_m''(s)\big )(e^{\la_{m,n}^+(T-s)}-e^{\la_{m,n}^-(T-s)})ds,\ \forall m,n, \\
 \chi_{m,n}^2& =& \frac{\xi_{m,n}}{2i\al_{m,n}\sinh(\kappa_{m}\pi)}
			\int_0^T\big ( f_m''(s)\big )
			(\la_{m,n}^+e^{\la_{m,n}^+(T-s)}-\la_{m,n}^-e^{\la_{m,n}^-(T-s)})ds,
			\forall m,n. 
		\end{eqnarray*}
For $j=1,2$, let  
$$\zeta_{m,n}^j=\frac{2i\al_{m,n}\sinh(\kappa_m\pi)
\chi_{m,n}^j}{\xi_{m,n}}.$$
Thus for fixed $m$, by \rq{est1} 
\beq
|\zeta_{m,n}^1|\leq C\om_{m,n}^{3/2}e^{-\rho \om_{m,n}T/2};\ 
|\zeta_{m,n}^2|\leq C\om_{m,n}^{5/2}e^{-\rho \om_{m,n}T/2}.\label{zeta}
\eeq

Then
we rewrite the equations above as
 \begin{eqnarray}
\tau_{m,n}^+:=\frac{\zeta_{m,n}^2-\la^-_{m,n}\zeta_{m,n}^1}{\al_{m,n}} & =& \int_0^T f_m''(s)e^{\la_{m,n}^+(T-s)}ds,\ \forall m,n, \label{md1p} \\
\tau_{m,n}^-:=\frac{\zeta_{m,n}^2-\la^+_{m,n}\zeta_{m,n}^1}{\al_{m,n}}& =& 
			\int_0^Tf_m''(s)
			e^{\la_{m,n}^-(T-s)}ds,
			\forall m,n. \label{md2p}
		\end{eqnarray}       
Thus 
$$
|\tau_{m,n}^\pm|\leq C\om_{m,n}^{3/2}e^{-\rho \om_{m,n}T/2},$$
and furthermore, there exists a constant independent of the initial conditions such that
\beq
\sum_\pm \sum_{m,n}|\tau_{m,n}^\pm|^2e^{\rho T\om_{m,n}}
\leq C(\| u^0\|^2_{X^3}+\| u^1\|^2_{X_1}).
  \label{tau}
\eeq      
Requiring $f\in H^2_0(0,T;L^2(M))$ is equivalent to stipulating that $f''$ also satisfy
\begin{eqnarray}
\tau_{m,0}^+:=0 &= & \int_0^T f_m''(s)ds,\label{md3p} \\
\tau_{m,0}^-:=0 &= & \int_0^T sf_m''(s)ds,\ \forall m.\label{md4p} 
\end{eqnarray}	
Our moment problem is thus \rq{md1p}-\rq{md4p}.
	
We now solve for  $f''$, hence for $f$. 
Fix $m$. We adapt the notation to Theorem \ref{win}.
$$
		\la_{k,m} =\left \{
		\begin{array}{cc}
			-i\la_{k,m}^+, & k> 0,\\
			-i\la_{|k|,m}^-, & k<0,
		\end{array}
		\right .
$$
We then set $\la_{0,m}=0$, and 
consider the exponential family in $L^2(0,T)$:
$$
{\cal E}_m =\{ s, e^{s\la_{k,m}},\ k\in \bZ\}.
$$
It is easy to verify that there exists $\delta>0$, independent of $m$, such that 
$$
k\neq k'\mbox{ implies }\ |\la_{k',m}-\la_{k,m}|>\delta .
$$
Furthermore, in the notation of Section \ref{wins},  $\nu_m (r)$ can be chosen so
$$C_1r^{1/2}<\nu_m (r)<C_2r^{1/2}$$
with $C_1,C_2$ positive constants independent of $m$.
 Thus there exists a family of functions $\{  g_{0,2,m}(t),g_{l,m}(t);\ l\in \bZ , \} $ in $L^2(0,T)$ satisfying 
$$  \< g_{l,m},t\> =0,\ \< g_{l,m},e^{i\la_{k,m}t}\> =\delta_{l,k},\ \< g_{0,2,m},t\>=1,\  \< g_{0,2,m},e^{i\la_{k,m}t}\> =0 ,\ \forall m.
$$
Furthermore, by Theorem \ref{win}, part B, there exist positive constants $C_2,C_3$ depending  only on $R_0,T,C_0,C_1$ such that
$$
\| g_{l,m}\|_{L^2(0,T)}\leq C_2 e^{C_3(\Im (\la_{l,m}))^{1/2}}.
$$

Set $\tau_{m,0}^\pm =0$, $\tau_{m,k}=\tau^+_{m,k}$ for $k>0$, and $\tau_{m,k}=\tau^-_{m,k}$ for $k<0$.
Then formally, our moment problem for fixed $m$ is solved by 
$$f_m''(t)=\sum_{n\in \bK}\tau_{m,k} g_{k,m}(t),$$
and 
$$f''(y,t)=\sum_{m\geq 1}
\sum_{n\in \bK}\tau_{m,k} g_{k,m}(t)
\phi_m(y).$$
We now verify the convergence of the last series. Applying Parseval's equation, followed by Theorem \ref{win}, Part B, then \rq{tau}, we get
\begin{eqnarray*}
\int_{\Gamma_1}\int_0^T |\sum_{m\geq 1}\sum_{k\in \bK}\tau_{m,k} g_{k,m}(t)\phi_m(y)|^2dydt
&\leq & \int_M\int_0^T |\sum_{m\geq 1}\sum_{k\in \bK}\tau_{m,k} g_{k,m}(t)\phi_m(y)|^2dydt\nonumber\\
& = & \sum_m \int_0^T |\sum_{k\in \bK}\tau_{m,k}^\pm g_{k,m}(t)|^2dt\\
& \leq & \sum_m \sum_{k\in \bK}|\tau_{m,k} |^2\int_0^T |g_{k,m}(t)|^2dt\\
& \leq & \sum_m \sum_{k\in \bK}|\tau_{m,k}|^2C_2e^{2C_3\Re (-\la_{k,m})^{1/2}T}\nonumber\\
& \leq & \sum_m \sum_{k\in \bK}|\tau_{m,k} |^2C_2e^{2C_3\om_{m,|k|}^{1/2}}\\
& \leq & C(\| u^0\|^2_{X^3}+\| u^1\|^2_{X_1}).
\label{per}
\end{eqnarray*}
 Finally,
$$
f(t)=\int_{s=0}^t\int_{r=0}^sf''(r)\ dr\ ds.\Box
$$

\section{Appendix}
Define an orthonormal basis of $L^2(0,2\pi)$ by 
$$
\phi_n(x)=\left \{
\begin{array}{cc}
1/\sqrt{2\pi} & n=0,\\
\cos{nx}/\sqrt{\pi} & n>0,\\
\sin{nx}/\sqrt{\pi} & n>0.
\end{array}
\right .
$$

{\bf Proof of Proposition \ref{UMa}}: It can be shown, adapting the proof  (\cite{AEI}, Lemma 4), that the angle between 
$\phi_n,\phi_m$ in $L^2(a,b)$ is bounded away from zero if $n\neq m$. It follows that the angle between $\phi_n$ and the closure of the span of $\{ \phi_m: m\neq n\}$ in $L^2(a,b)$ is bounded below away from zero. 
It then follows, see \cite{AI}, that
$\{ \phi_n\}$ is uniformly minimal, i.e. it admits a biorthogonal family 
$\{ g_n(x): n\in \bZ\}$, and there exists a constant $M>0$ such that 
\beq
\| g_n\|_{L^2(a,b)}<M,\ \forall n.\label{geq}
\eeq
Let $\ell^2$ be the set of square summable sequences of complex numbers parametrized by $\bZ$, with inner product $\< *,*\>$ and norm $\| *\|_{\ell^2}$. Let $(*,*) $ be the
scalar product in $L^2(a,b)$, with corresponding norm $\| *\|_{\ell^2}$. Let $\| *\|$ be the operator norm for the appropriate Hilbert spaces. 
Let 
$f=\sum_{n\in \bZ}c_{n}\phi_n(x)$ be a finite sum, and let $Lf=\{ c_n/h(n)\}\in \ell^2$.
Then 
\begin{eqnarray}
\| Lf\|_{\ell^2}^2 & = & \sum_n |c_n|^2/h(n)^2\nonumber \\
& = & \sum_n \bar{c}_n(\sum_m \frac{c_m}{h(m)^2}g_m,\phi_n)\nonumber\\
& = &  ( \sum_m \frac{c_m}{h(m)^2}g_m,f)\nonumber\\
& \leq & \|f\|_2 \ \| \sum_m \frac{c_m}{h(m)^2}g_m\|_2.\label{A}
\end{eqnarray}

Let us introduce the Gramian matrix $\Gamma$: $l^2 \mapsto l^2$, with entries
$$
\Gamma_{n;m}= (  g_{n}/h(n), g_{m}/h(m)).
$$
Then, letting $Lf=\vec{c}=\{ c_{n}/h(n)\},$
\begin{eqnarray}
\|\sum \frac{c_n}{h(n)^2} g_{n}\|^2_2
& =& \sum_{m}c_{m}/h(m)\sum_{n}\Gamma_{n,m}\bar{c}_{n}/h(n)\nonumber \\
&=&  \< \vec{c},\overline{\Gamma\bar{\vec{c}}}\>\nonumber\\
& \leq & \|\Gamma\|  \|Lf\|_{\ell^2}^2.\label{B}
\end{eqnarray}
Combining \rq{A},\rq{B}, we obtain 
$$\|L\|\le\|\Gamma \|^{1/2}.$$
By the Gershgorin Theorem, 
$$
\|\Gamma  \|\leq \sup_{m}\{ \frac{1}{h(m)}\sum_{n}\frac{1}{h(n)}( g_{m},g_{n})\}.
$$
Thus there exists a constant $C>0$ such that, by  \rq{geq},
\begin{eqnarray*}
\sum_{n}\frac{1}{h(n)}|( g_{m},g_{n})|& \leq & 
CM^2.
\end{eqnarray*}
This proves the proposition for $f$ a finite sum. The proof is completed by 
  a density argument. $\Box$

\medskip

Finally, the proof of Proposition \ref{UM} is similar but easier than the proof that we just presented, so is left to the reader. 


\begin{thebibliography}{99}



\bibitem{AD} N. Abatangelo, L. Dupaigne,   ``Nonhomogeneous boundary conditions for the spectral fractional
Laplacian", Ann. I. H. Poincaré – AN 34 (2017) 439–467.




\bibitem{AE}
\newblock S.  Avdonin and J. Edward,
``Boundary null-controllability for the beam equation with classical structural damping", in submission.


\bibitem{AEI}
 Avdonin, Sergei; Edward, Julian; Ivanov, Sergei. { ``Null-controllability for the beam equation with structural damping. 
Part 1. Distributed control."} J. Differential Equations 421 (2025), 73–103.

\bibitem{AEI2} 
 Avdonin, Sergei; Edward, Julian; Ivanov, Sergei. 
Null-controllability for the beam equation with fractional
structural damping. Part 2. Boundary control. Preprint.

 
 
 



\bibitem{AI} 
\newblock S. A. Avdonin and S. A. Ivanov,
\newblock \emph{Families of Exponentials. The Method of Moments in Controllability Problems for Distributed Parameter Systems},
\newblock Cambridge University Press, New York, London, Melbourne, 1995.

\bibitem{AL}
G. Avalos and I. Lasiecka, "Optimal blowup rates for the minimal
energy null control of the strongly damped abstract wave
equation".  Ann. Sc. Norm. Super. Pisa Cl. Sci. (5) 2 (2003), no.
3, 601--616.


\bibitem{APR}  H. Antil, J. Pfefferer and S. Rogovs,
``Fractional Operators with  inhomogeneous Boundary
Conditions: Analysis, Control, and Discretization",
Commun. Math. Sci.  (2018),
Vol. 16, No. 5, pp. 1395–1426


\bibitem{Ba} A.V. Balakrishnan, ``Damping operators in continuum models of flexible structures: Explicit models for proportional damping in beam bending with end-bodies", Appl. Math. Optim. 21 (3) (1990) 315–334.

\bibitem{B}
F. Boyer,  ``Controllability of linear parabolic equations and systems". Master. France. 2022.
hal-02470625v4


\bibitem{CR} G. Chen and D.L. Russell, ``A mathematical model for
linear elastic systems with structural damping", Quart. Appl.
Math., 39, (1982), 433-454.

\bibitem{CS} Luis A. Caffarelli and Pablo Raúl Stinga, ``Fractional elliptic equations, Caccioppoli estimates and regularity",
 Annales de l'Institut Henri Poincaré C, Analyse non linéaire
Volume 33, Issue 3,  (2016), Pages 767-807.

\bibitem{D}  E.B. Davies, {\it Heat Kernels and Spectral Theory}, Cambridge Tracts in Mathematics, vol. 92, Cambridge University Press, Cambridge, 1989.


\bibitem{E} Edward, J. ``Complex Ingham type inequalities and applications to control theory",
Journal of Mathematical Analysis and Applications, 324 (2006)

\bibitem{E2} Edward. ``An extension of the spectral fractional Laplacian to non-homogeneous boundary condition on rectangular domains, with application to well-posedness for plate equation with structural damping", J. Math. Anal. Appl. 556 (2026), no. 1, part 1, Paper No. 130073, 22 pp. 

\bibitem{ET} Edward, J. and Tebou, L., ``Uniform internal controllability for structurally damped beam equation", Asymptotic Analysis, 47 (2006), 55-83.

\bibitem{H} S.W. Hansen, "Bounds on functions biorthogonal to sets of
complex exponentials; control of elastic damped systems", J. Math.
Anal. Appl. 158 (1991), 487-508

\bibitem{H2} Hansen, S. W.,
``Optimal regularity results in boundary control of elastic systems with fractional order damping", Contrôle des systèmes gouvernés par des équations aux dérivées partielles (Nancy, 1999), 53–64.
ESAIM Proc., 8 (2000)
Société de Mathématiques Appliquées et Industrielles, Paris.



\bibitem{K} V. Komornik and P. Loreti, {\em Fourier series in
control theory}, Springer Monographs in Mathematics, 2005.



\bibitem{LT} I. Lasieka and R. Triggiani,``Exact null-controllability of
structurally damped and thermoelastic parabolic models", Rend.
Mat. Acc. Lincet, s.9, v.9 (1998), p.43-69.



\bibitem{LR} G. Lebeau and L. Robbiano, ``Controle exact de l’equation de la chaleur", ´
Comm. Partial Diff. Eqs., 20 (1995), 335–356.

\bibitem{LZ} G. Lebeau and E. Zuazua,
``Null-Controllability of a System
of Linear Thermoelasticity",
Arch. Rational Mech. Anal. 141 (1998) 297–329.


\bibitem{lions1} J.L. Lions, {\em Controlabilit\'e exacte, perturbations, et stabilisation de syst\`emes distribu\'es, Tome1: Controlabilit\'e Exacte}, Recherches en mathematiques appliqu\'ees, 8. Masson, Paris 1988.



\bibitem{L} Lischke, Anna; Pang, Guofei; Gulian, Mamikon; et al.
``What is the fractional Laplacian? A comparative review with new results",
J. Comput. Phys. 404 (2020), 109009, 62 pp.

\bibitem{mil} Miller, Luc ``Non-structural controllability of linear elastic systems with structural damping", J. Funct. Anal. 236 (2006), no. 2, 592–608.


\bibitem{mit} Mitra, Sourav,  ``Carleman estimate for an adjoint of a damped beam equation and an application to null controllability". J. Math. Anal. Appl. 484 (2020), no. 1, 123718, 29 pp.

\bibitem{AIS} Seidman, T. I., Avdonin, S. A.; Ivanov, S. A., ``The "window problem'' for series of complex exponentials". J. Fourier Anal. Appl. 6 (2000), no. 3, 233–254.

\bibitem{SV}  R. Song, Z. Vondracek, ``Potential theory of subordinate killed Brownian motion in a domain", Probab. Theory Relat. Fields 125 (4) (2003)
578–592.


\bibitem{T1} R. Triggiani, 
``Regularity of Some Structurally Damped 
Problems with Point Control 
and with Boundary Control", Journal of Mathematical Analysis and Applications,  161, (1991) 299933


\bibitem{T2} R. Triggiani, "Optimal estimates of norms of fast controls in exact null controllability of two non-classical abstract parabolic systems". Adv. Differential Equations 8 (2003), no. 2, 189--229.





\bibitem{T}  R. Triggiani,
"Optimal quadratic boundary control problem for wave- and
plate-like equations with high internal damping: an abstract
approach." (English. English summary) With an appendix by Dahlard
Lukes. Lecture Notes in Pure and Appl. Math., 165, Control of
partial differential equations (Trento, 1993), 215--270, Dekker,
New York, 1994.



\bibitem{Y} R. M. Young, {\em An introduction to non-harmonicFourier series.  Revised first edition}, Academic Press, SanDiego, 2001.

\bibitem{Z2} E. Zuazua,``Exact controllability for semilinear wave equations
in one space dimension",
Annales de l’I. H. P., section C, tome 10, no 1 (1993), p. 109-129.



\bibitem{Z} E. Zuazua, "Exact Controllability for distributed systems for arbitrarily small time",
Proceedings of the 26th IEEE Conference on Decision and Control,
Los Angeles, 1987.



\end{thebibliography}
\end{document}